\documentclass[preprint,12pt,authoryear]{elsarticle}
\setcitestyle{numbers,square,comma}

\usepackage{amssymb}
\usepackage{verbatim}
\usepackage{tikz}
\usepackage{tikz-cd}
 \usepackage{amsthm}
 \usepackage{amsmath}
 \usepackage{hyperref}
 \usepackage{multirow}
 \usepackage{longtable}
 \hypersetup{hidelinks}
 \usepackage[font=small,labelfont=bf,labelsep=none]{caption}

\newcommand{\fq}{\mathbb{F}_q}

\journal{XX}

\begin{document}

\begin{frontmatter}



\title{A classification of permutation binomials of the form $x^i+ax$ over $\mathbb{F}_{2^n}$ for dimensions up to $8$}


\author[1,2]{Yi Li, Xiutao Feng}\author[3]{Qiang Wang}

\affiliation[1]{organization={Key Laboratory of Mathematics Mechanization, Academy of Mathematics and Systems Science, Chinese Academy of Sciences},
	city={Beijing},
	postcode={100190}, 
	country={China}}

\affiliation[2]{organization={University of Chinese Academy of Sciences},
	city={Beijing},
	postcode={100049},
	country={China}
}
\affiliation[3]{organization={School of Mathematics and Statistics, Carleton University},city={Ottawa, Ontario},
postcode={K1S 5B6},
country={Canada}}

\begin{abstract}
 Permutation polynomials with few terms (especially permutation binomials) attract many people due to their simple algebraic structure. Despite the great interests in the study of permutation binomials, a complete characterization of permutation binomials is still unknown. In this paper, we give a classification of permutation binomials of the form $x^i+ax$ over $\mathbb{F}_{2^n}$, where $n\leq 8$ by characterizing three new classes of permutation binomials. In particular one of them has relatively large index $\frac{q^2+q+1}{3}$ over $\mathbb{F}_{q^3}$. 

\end{abstract}

\begin{keyword}
permutation binomials \sep classification    \sep Hermite's criterion \sep the AGW criterion


\MSC[2020] 11T06\sep 11T55

\end{keyword}

\end{frontmatter}


\newtheorem{rmk}{Remark}
\newtheorem{thm}{Theorem}
\newtheorem{case}{Case}		
\newtheorem{subcase}{Case}
\newtheorem{defn}[thm]{Definition}
\numberwithin{subcase}{case}
\newtheorem{lem}[thm]{Lemma}
\newtheorem{cor}[thm]{Corollary}

\section{Introduction}

Let $q=p^n$ be a prime power, $\mathbb{F}_{q}$ denote the finite field with $q$ elements. A polynomial $f\in \mathbb{F}_{q}[x]$ is called a $permutation \hspace*{0.4em}polynomial$ (PP) over $\mathbb{F}_q$ if the associated polynomial function $f:c\rightarrow f(c)$ from $\mathbb{F}_q$ into $\mathbb{F}_q$ is a permutation. Permutation polynomial is an active research area because of its applications in coding theory \cite{OTCCFM2013,PPACC2007}, cryptography \cite{RSA1978,PKE1998} and combinatorial designs \cite{AFSHDS2006}, and we refer the reader to \cite{PPOFFASORA2015,HFF2013,PPOFFAIA2019} for more details of the recent advances and contributions on them. 

Permutation polynomials with few terms attract many people due to their simple algebraic structure. It is well known that $x^i$ is a permutation monomial if and only if $\text{gcd}(i,q-1)=1$.  In 1993, 
 Lidl and Mullen proposed an open problem (Problem 3.1) in \cite{WDAPOAFFPTEOTF1993}:
determine conditions on $k$, $r$, and $q$ so that
$P(x) = x^k + ax^r$ permutes $\mathbb{F}_{q}$ with  $a \in\mathbb{F}_{q}^*$.
There are many papers on permutation binomials published in the past thirty years. In particular, different types of characterizations were given.  We refer the reader to 
 \cite{BPOFFWEC2021,NCOPBAPTOFF2017,ACOPBOFF2013,DOATOPBOFF2015,ANRPBO2015,FROCPMOFF2019,CMBF2008,PACPP2015,OOCOPPOFFCT2015,SCOMCPPOFFOCT2014,SCOPPOF2015,OMCPP2016,CPPFEP2017,PPOFFH2016,OSPBATO2017,SCOCPP2014,SNROPPOFF2017,PBOFF2022,NROPBPFF2023,NASCOTCOPP2022,PBOTF2022,SCOCOFFOEC2020,CPPOFFOOC2015,PPIFPOSASCSOMOLS2013,AGLSAPB2006,STOPP1962,TNOSOASSOEIAFF1966,CMOFF1982,PB1990,PPOFF1987,PBOFF1994,OSPPOFF2005,OPOTF2007,BPOFF1988,OPPOFF1987,AFF1991,PPOPT,OCBOFF2005,PPOFTFATGS1991,CMPPOFF2007,OPP2002,SFOPPOFF2008}. Despite the great interests in the study of permutation binomials,  a complete characterization of permutation binomials is still unknown.


Motivated by these results, we perform an exhaustive search on permutation  binomials of the form $P(x) = x^i+ax$ over $\mathbb{F}_{q}$, where $q=2^n$ such that $n\leq 12$.  Let $s=\gcd(i-1, 2^n-1)$  and $d=\frac{q-1}{s}$. Then $P(x)= x(x^{es}+a)$ for some positive integer $e$.  Here $d$ is called the index of $P(x)$ \cite{OPPOPS2009}. In general, the smaller their indices, the simpler are these PPs. Indeed, we can compute all the coefficients of these PPs of small indices using Inverse Discrete Fourier Transform. We refer to \cite{PPOFFAIA2019} for more results and details on the index of any polynomial over a finite field $\fq$. Also in \cite{PPOFFAIA2019}, Wang proposed to classify permutation polynomials by their indices and we follow this approach to study  permutation binomials of the form $x^i+a x$ over $\mathbb{F}_{2^n}$ in terms of their indices. 
Since it is easy to determine when a linearized binomial $x^{2^j} + ax$ is a PP over $\mathbb{F}_{2^n}$ (i.e.,  $a$ is not a ($2^{\gcd(j, n)}-1$)-th power), we focus on non-linearzied permutation binomials from now on. 

Below are some tables of non-linearized permutation binomial $x^i+ax$ over $\mathbb{F}_{2^n}$, where $6\leq n\leq 12$. In each table, we provide the values of $i$ such that $x^i+ax$ is permutation binomial for some $a\in \mathbb{F}_{2^n}^*$, the indices of corresponding permutation binomials, as well as the references where we can use general classes of PPs to explain these examples. In some cases, we don't have explicitly characterized classes to include these examples, instead, we provide some references for general implicit characterizations. 
We observe that most of non-linearized permutation binomials over $\fq$  have small indices, which are less than $\sqrt{q}+1$. However,  we find three permutation binomials over $\mathbb{F}_{2^{12}}$ whose indices are larger than 65, that is,  $x^{136}+ax$, $x^{271}+ax$ and $x^{1846}+ax$, which have an index $91$ and they are listed in Table \ref{Table5}. In contrast,  the indices of permutation binomials over prime fields are always less than $\sqrt{q}+1$ (see \cite{PBOFF2009}). This indicates the different behaviors of permutation binomials between prime fields and non prime fields.

Motivated by these examples, we characterize three new classes of permutation binomials over $\mathbb{F}_{2^n}$ and use them to complete the classification  of all permutation binomial of the form $x^i+ax$ over $\mathbb{F}_{2^n}$ where $n \leq 8$.
We also note that one of these new classes provides permutation binomials with relatively large  indices $\frac{q^2+q+1}{3}$ over $\mathbb{F}_{q^3}$, which can explain two of three permutation binomials of index $91$ over $\mathbb{F}_{2^{12}}$.  This also provides a new class of PPs to answer Problem 7 in \cite{PPOFFAIA2019}. 

We note that  there are no permutation binomials of the form $x^i+ax$ in $\mathbb{F}_{2^5}$, $\mathbb{F}_{2^7}$, and $\mathbb{F}_{2^{11}}$.
Indeed, using Corollary~3 in \cite{TNOFPB2006},  we can explain that there  is no permutation binomial  over $\mathbb{F}_{2^5}$ or $\mathbb{F}_{2^7}$ because $31,127$ are Mersenne primes. However, $2^{11}-1=23\times89$ is not a prime. So, it still remains to answer  why there is no permutation binomial of the form $x^i+ax$ over $\mathbb{F}_{2^{11}}$. 

The rest of the paper is organized as follows. In Section 2, we introduce some preliminary results and  techniques such as Hermite's criterion and the AGW criterion.   In Section 3, we state and prove three new classes of permutation binomials $x^i+ax$ over $\mathbb{F}_{q^2}$, $\mathbb{F}_{q^3}$, and $\mathbb{F}_{q^4}$ respectively and a nonexistence result.  The discussion in Section 4 provides a complete explanation of all permutation binomials of the form $x^i+ax$ over $\mathbb{F}_{2^n}$ where $n\leq 8$. The conclusion and future work are given in Section 5. 

\begin{table}[h]
    \centering
     \caption{\hspace{0.5em}Non-linearized permutation binomials $x^i+ax$ over $\mathbb{F}_{2^6}$}
    \label{Table1}
    \begin{tabular}{|c|c|c|}
         \hline
         $i$&Index&Refs\\
         \hline
         10&7&\textbf{this paper}, \cite{PPIFPOSASCSOMOLS2013}\\
         \hline
         19&7&\textbf{this paper}, \cite{PPIFPOSASCSOMOLS2013}\\
         \hline
         22&3&\cite{DOATOPBOFF2015}\\
         \hline
         43&3&\textbf{this paper}\\
         \hline
         
    \end{tabular}
   
\end{table}

\begin{table}[h]
    \centering
     \caption{\hspace{0.5em}Non-linearized permutation binomials $x^i+ax$ over $\mathbb{F}_{2^8}$}
    \label{Table2}
    \begin{tabular}{|c|c|c|}
         \hline
         $i$&Index&Refs\\
         \hline
         86&3&\cite{OOCOPPOFFCT2015}\\
         \hline
         154&5&\textbf{this paper}\\
         \hline
         171&3&\textbf{this paper}\\
         \hline
    \end{tabular}
   
\end{table}

\begin{table}[h]
    \centering
     \caption{\hspace{0.5em}Non-linearized permutation binomials $x^i+ax$ over $\mathbb{F}_{2^9}$}
    \label{Table3}
    \begin{tabular}{|c|c|c|}
         \hline
         $i$&Index&Refs\\
         \hline
         74&7&\cite{PACPP2015,PPIFPOSASCSOMOLS2013}\\
         \hline
         366&7&Implicit characterization \cite{PPOFTFATGS1991, CMPPOFF2007,PPOFFAIA2019}\\
         \hline
    \end{tabular}   
\end{table}

 \begin{table}[h]
    \centering
     \caption{\hspace{0.5em}Non-linearized permutation binomials $x^i+ax$ over $\mathbb{F}_{2^{10}}$}
    \label{Table4}
    \begin{tabular}{|c|c|c|c|c|}
         \hline
         $i$&34&67&94&187\\
         \hline
         Index&31&31&11&11\\
         \hline
         Refs&\cite{PACPP2015,PPIFPOSASCSOMOLS2013}&\cite{DOATOPBAT2020,PPIFPOSASCSOMOLS2013}&\cite{DOATOPBOFF2015}&\textbf{this paper}\\
         \hline
         $i$&280&331&342&397\\
         \hline
         Index&11&31&3&31\\
         \hline
         Refs&\multicolumn{4}{|c|}{Implicit characterization \cite{PPOFTFATGS1991,CMPPOFF2007,PPOFFAIA2019}}\\
         \hline
         $i$&466&559&652&683\\
         \hline
         Index&11&11&11&3\\
         \hline
         Refs&\multicolumn{4}{|c|}{Implicit characterization \cite{PPOFTFATGS1991,CMPPOFF2007,PPOFFAIA2019}}\\
         \hline
         $i$&745&838&931&\\
         \hline
         Index&11&11&11&\\
         \hline
          Refs&\multicolumn{3}{|c|}{Implicit characterization \cite{PPOFTFATGS1991,CMPPOFF2007,PPOFFAIA2019}}&\\
          \hline
    \end{tabular}
   
\end{table}

 \begin{table}[h]
    \centering
     \caption{\hspace{0.5em}Non-linearized permutation binomials $x^i+ax$ over $\mathbb{F}_{2^{12}}$}
    \label{Table5}
    \begin{tabular}{|c|c|c|c|c|}
         \hline
         $i$&136&271&274&316\\
         \hline
         Index&91&91&15&13\\
         \hline
         Refs&\textbf{this paper}&\textbf{this paper}&\cite{PACPP2015,PPIFPOSASCSOMOLS2013}&\cite{ANRPBO2015}\\
         \hline
         $i$&547&586&631&820\\
         \hline
         Index&15&7&13&5\\
         \hline
         Refs&\multicolumn{4}{|c|}{Implicit characterization \cite{PPOFTFATGS1991, CMPPOFF2007,PPOFFAIA2019}}\\
         \hline
         $i$&946&1093&1171&1260\\
         \hline
         Index&13&15&7&13\\
         \hline
         Refs&\multicolumn{4}{|c|}{Implicit characterization \cite{PPOFTFATGS1991, CMPPOFF2007,PPOFFAIA2019}}\\
         \hline
         $i$&1366&1576&1639&1846\\
         \hline
         Index&3&15&5&91\\
         \hline
         Refs&\multicolumn{4}{|c|}{Implicit characterization \cite{PPOFTFATGS1991,CMPPOFF2007,PPOFFAIA2019}}\\
         \hline
         $i$&1890&2146&2206&2276\\
         \hline
         Index&13&21&15&9\\
         \hline
         Refs&\multicolumn{4}{|c|}{Implicit characterization \cite{PPOFTFATGS1991,CMPPOFF2007,PPOFFAIA2019}}\\
         \hline
         $i$&2341&2458&2521&2536\\
         \hline
         Index&7&5&13&21\\
         \hline
         Refs&\multicolumn{4}{|c|}{Implicit characterization \cite{PPOFTFATGS1991,CMPPOFF2007,PPOFFAIA2019}}\\
         \hline
         $i$&2731&2836&3004&3151\\
         \hline
         Index&3&13&15&13\\
         \hline
         Refs&\multicolumn{4}{|c|}{Implicit characterization \cite{PPOFTFATGS1991,CMPPOFF2007,PPOFFAIA2019}}\\
         \hline
         $i$&3277&3466&3511&3781\\
         \hline
         Index&5&13&7&13\\
         \hline
         Refs&\multicolumn{4}{|c|}{Implicit characterization \cite{PPOFTFATGS1991,CMPPOFF2007,PPOFFAIA2019}}\\
         \hline
    \end{tabular}
   
\end{table}

\section{Preliminaries}

The following Hermite's criterion is one of the most useful tools to characterize permutation polynomials. 

\begin{lem}\cite{H1863,FF1997} (Hermite's criterion)
    Let $\mathbb{F}_q$ be a finite field of characteristic p. Then $f\in \mathbb{F}_{q}[x]$ is a permutation polynomial of $\mathbb{F}_{q}$ if and only if the following two conditions hold:

    (i) $f$ has exactly one root in $\mathbb{F}_{q}$;

    (ii) for each integer $t$ with $1\leq t\leq q-2$ and $t\not\equiv 0\hspace{0.5em}(\text{mod}  \hspace{0.5em}p)$, the reduction of $f(x)^t\mod (x^q-x)$ has degree $\leq q-2$.
\end{lem}

To apply Hermite's criterion on permutation binomials, we need to show the coefficient of $x^{q-1}$ in the reduction of $f(x)^t \mod (x^q-x)$ vanishes.  The following result is very useful  to compute binomials coefficients over extension fields.  

\begin{lem}(Lucas' lemma) 
    Let $n,i$ be positive integers and $p$ be a prime. Assume $n=a_mp^m+a_{m-1}p^{m-1}+...+a_1p+a_0$ and $i=b_mp^m+...+b_1p+b_0$, where $m$ is a non-negative integer and $0\le a_i, b_i<p$ for $i=0,1,\cdots, m$. Then 
    \[
    \binom{n}{i}\equiv\binom{a_m}{b_m}\binom{a_{m-1}}{b_{m-1}}...\binom{a_0}{b_0} \hspace{0.5em}(\text{mod} \hspace{0.5em}p).
    \]
\end{lem}

In recent years, the AGW criterion has become another important tool to characterize and construct new classes of permutation polynomials. 

\begin{lem}\cite{OCPOFF2011} (The AGW criterion)
    Let $A,S$ and $\bar{S}$ be finite sets with $|S|=|\bar{S}|$. Let $f, \bar{f}, \lambda, \bar{\lambda}$ be maps on finite sets such that $f: A\rightarrow A$, $\bar{f}:S\rightarrow \bar{S}$, $\lambda: A\rightarrow S$, $\bar{\lambda}: A\rightarrow \bar{S}$, and $\bar{\lambda}\circ f=\bar{f}\circ \lambda$.
\begin{center}
    \begin{tikzcd}
A \arrow[rr, "f"] \arrow[dd, "\lambda"] &  & A \arrow[dd, "\bar{\lambda}"] \\
                                        &  &                               \\
S \arrow[rr, "\bar{f}"]                 &  & \bar{S}                      
\end{tikzcd}
\end{center}

If $\lambda$ and $\bar{\lambda}$ are surjective, then the following are equivalent:

(i) $f$ is a bijection from $A$ to itself;

(ii) $\bar{f}$ is a bijection from $S$ to $\bar{S}$ and $f$ is injective on $\lambda^{-1}(s)$ for each $s\in S$.
\end{lem}

In particular, if we let $A=\mathbb{F}_{q}^{*}$ and the maps $\lambda=\bar{\lambda}=x^s$, where $q-1=ds$ for some positive integers $s,d$, then we obtain the so-called multiplicative case of the AGW criterion.

\begin{cor}\cite{SPTOFF1997,OSPPO2009,CMPPOFF2007,PPOFTFATGS1991}
    Let $q-1=ds$, where $d$ and $s$ are two positive integers and $q$ is a prime power. Then $p(x)=x^rf(x^s)$ is a permutation polynomial over $\mathbb{F}_{q}$ if and only if $\text{gcd}(r,s)=1$ and $x^rf(x)^s$ permutes the set $\mu_{d}$ of $d$-th roots of unity.
\end{cor}

As for some special form of permutation binomials, some characterizations were given. For example, a characterization of permutation binomials of the form $x^{1+\frac{q-1}{m}}+ax$ over $\mathbb{F}_q$ was given in \cite{CMOFF1982}. For the polynomials which have low indices, Wang \cite{PPOFFAIA2019} proposed an algorithm to classify all PPs over $\mathbb{F}_{q}$ of small indices explicitly in terms of their coefficients.
These results can make our tables more accurate.

An exceptional polynomial over $\mathbb{F}_q$ is a polynomial $f\in \mathbb{F}_{q}[x]$ which is a permutation polynomial on $\mathbb{F}_{q^m}$ for infinitely many $m$. In the sequel, we also need some results on exceptional polynomials.  

\begin{lem}\cite{HFF2013}\label{PPEP}
    A permutation polynomial over $\mathbb{F}_{q}$ of degree at most $q^{\frac{1}{4}}$ is exceptional over $\mathbb{F}_{q}$.
\end{lem}

\begin{lem}(Carlitz-Wan conjecture/theorem)\cite{HFF2013,PFCW11995,EPOAT1997}
    Exceptional polynomials over $\mathbb{F}_q$ have degree coprime to $q-1$. 
\end{lem}

\section{Three new classes of permutation binomials  and a nonexistence result}

Let $q=2^n$. In this section we characterize three  new classes of permutation binomials $x^i+ax$ over $\mathbb{F}_{q^2}$, $\mathbb{F}_{q^3}$, and $\mathbb{F}_{q^4}$ respectively.  We also prove that there is no permutation binomial of the form $x^{2(q^3+q^2+q+1)+1} +a x$ over $\mathbb{F}_{q^4}$ with index $q-1$, where $m$ even and $m\geq 4$.  

\subsection{Permutation binomials over $\mathbb{F}_{q^2}$}

In \cite{PPOFIFBRFOSOTMG2013}, Zieve proved $x^{r(q-1)+1}+ax$ is a PP of $\mathbb{F}_{q^2}$ if and only if $(-a)^{\frac{q+1}{\text{gcd}(r,q+1)}}\neq 1$ and $\text{gcd}(r-1,q-1)=1$ under the assumption that $a^{q+1}=1$. In \cite{ACOPBOFF2013,DOATOPBOFF2015,ANRPBO2015}, Hou and Lappano used Hermite’s criterion to determine several explicit classes of permutation binomials with the forms $x^{2q-1}+ax$, $x^{3q-2}+ax$, $x^{5q-4}+ax$ and $x^{7q-6}+ax$ over $\mathbb{F}_{q^2}$. Later, in \cite{PPOFFH2016}, Hou proved that let $r>2$ be a fixed prime, under the assumption $a^{q+1}\neq 1(a\in \mathbb{F}_{q}^{*})$, there are only finitely many $(q,a)$ for which $x^{r(q-1)+1}+ax$ is a PP of $\mathbb{F}_{q^2}$.  Motivated by these results, Li et al \cite{NCOPBAPTOFF2017} completely characterized when $x^r(x^{q-1}+a)$ can be a permutation binomial. In \cite{BPOFFWEC2021}, sufficient conditions on $x^r(x^{3(q-1)}+a)$ were given. Very recently, in \cite{NROPBPFF2023}, Hou and Lavorante gave two nonexistence results about permutation binomials of the form $x^r(x^{d(q-1)}+a)$ over $\mathbb{F}_{q^2}$. 

The following new class of permutation binomials over $\mathbb{F}_{q^2}$ is motivated by the permutation binomial $x^{43}+ax$ over $\mathbb{F}_{2^6}$ and $x^{187}+ax$ over $\mathbb{F}_{2^{10}}$.  We use both Hermite's criterion and the AGW criterion in the proof.  

\begin{thm}\label{T3}
    Let $q=2^n$, $n\geq 3$ is odd, and $a\in \mathbb{F}_{q^2}^{*}$. Then the polynomial
    \[
    f_1(x)=x^{6q-5}+ax
    \]
    is a permutation polynomial in $\mathbb{F}_{q^2}$ if and only if one of the following occurs
    \begin{enumerate}
        \item $n\geq 5$, $a\in \mu_{q+1}$ and $a\not \in \mu_{\frac{q+1}{3}}$;
        \item $n=3$, $a\in \{\gamma^3, \gamma^6, \gamma^7, \gamma^{12}, \gamma^{14}, \gamma^{24}, \gamma^{27}, \gamma^{28}$, $\gamma^{33}, \gamma^{35}, \gamma^{45}, \gamma^{48}, \gamma^{49}, \gamma^{54}, \gamma^{56}$\}, where $\gamma$ is a primitive element in $\mathbb{F}_{2^6}$.
    \end{enumerate}
\end{thm}
\begin{proof}
    When $q=2^3$, it is easy to verify the conditions. Hence we assume $n\geq 5$ from now on. If $f_1(x)$ is a permutation polynomial, then by Hermite's criterion, we know for any $1\leq t\leq q^2-2$,  the coefficient of $x^{q^2-1}$ in the reduction of $f_1(x)^t \mod (x^{q^2} -x)$ vanishes. Let $t=3q-3$. We consider 
    \begin{eqnarray}\label{26}
        \begin{aligned}
    f_1(x)^t
        &=(x^{6q-5}+ax)^{3q-3}\\
        &=x^{3q-3}\sum\limits_{i=0}^{3q-3}\binom{3q-3}{i}x^{(6q-6)i}a^{3q-3-i}\\
        &=\sum\limits_{i=0}^{3q-3}\binom{3q-3}{i}a^{3q-3-i}x^{3q-3+(6q-6)i}.
        \end{aligned}
    \end{eqnarray}

It is clear to see that the coefficient of $x^{q^2-1}$ is $\sum\limits_{i=0}^{3q-3}\binom{3q-3}{i}a^{3q-3-i}$ such that $1+2i\equiv 0\hspace{0.5em}(\text{mod}\hspace{0.5em} \frac{q+1}{3})$. Next we determine  all  possible values of $i$ such that $1+2i\equiv 0\hspace{0.5em}(\text{mod}\hspace{0.5em} \frac{q+1}{3})$, that is,  
    \begin{eqnarray}\label{27}
        1+2i=k\frac{q+1}{3},
    \end{eqnarray}
where $k\in \mathbb{Z}$. From Equation (\ref{27}), we can get

\begin{eqnarray}\label{28}
    i=k\frac{q+1}{6}-\frac{1}{2}.
\end{eqnarray}

Because $i$ is a nonnegative integer and  $i\leq 3q-3$,  $k$ must be odd and $k\leq 17$. 
Plugging these 9 possible values of $k$ into Equation (\ref{28}), we obtain
{\small
\[
i \in  S:=
\left\{ 
\frac{q-2}{6}, \frac{q}{2}, \frac{5q+2}{6}, \frac{7q+4}{6}, \frac{9q+6}{6}, \frac{11q+8}{6}, \frac{13q+10}{6}, \frac{15q+12}{6}, \frac{17q+14}{6}
\right\}. 
\]
}

Hence the coefficient of $x^{q^2-1}$ in Equation (\ref{26}) is 
\begin{eqnarray}\label{29}
    \sum\limits_{i\in S}\binom{3q-3}{i}a^{3q-3-i}=0.
\end{eqnarray}

We note that $3q-3=2(q-1)+(q-1)=2q+\frac{q}{2}+\frac{q}{4}+...+4+1$. Using  $q=2^n$,  by Lucas' lemma,  $\binom{3q-3}{i}\neq 0$ if and only if the $(n+1)$th bit and second lowest bit of the binary expansion of $i$ are $0$. We denote by $S_0\subseteq S$ the subset of $S$ such that the $(n+1)$th bit and second lowest bit of the binary expansion of $i$ are $0$. 

Next we explicitly determine $S_0$.  Obviously $\frac{q}{2} \in S_0$. On the other hand,  neither $\frac{9q+6}{6}=q+\frac{q}{2}+1$ nor $\frac{15q+12}{6}=2q+\frac{q}{2}+2$ belongs to $S_0$. Then we look at $2$-adic expansion of $\frac{q-2}{6}$, that is $\frac{q-2}{6}=\frac{q}{8}+\frac{q}{32}+...+4+1$, which implies that $\frac{q-2}{6} \in S_0$. Because $\frac{q+1}{3}=2 (\frac{q-2}{6})+1 =\frac{q}{4}+\frac{q}{16}+...+8+2+1$, we obtain $\frac{5q+2}{6} \not\in S_0$. This also implies that $\frac{11q+8}{6}=q+\frac{q}{2}+1+\frac{q+1}{3} \not\in S_0$.   Similarly, $\frac{7q+4}{6}\not \in S_0$ because $\frac{7q+4}{6}=q+\frac{q+4}{6}=q+1+\frac{q-2}{6}=q+\frac{q}{8}+\frac{q}{32}+...+4+2$.  Moroever,  $\frac{13q+10}{6}\not \in S_0$ because $\frac{13q+10}{6}=2q+2+\frac{q-2}{6}$. Finally 
$\frac{17q+14}{6}=2q+\frac{q}{2}+\frac{q+1}{3}+2$ implies that $\frac{17q+14}{6} \in S_0$. Hence, we know $S_0=\{\frac{q}{2},\frac{q-2}{6}, \frac{17q+14}{6}\}$.

Hence the coefficient of $x^{q^2-1}$ in Equation (\ref{26}) is 

\begin{eqnarray}\label{30}
\begin{aligned}
 \sum\limits_{i\in S}\binom{3q-3}{i}a^{3q-3-i} & = \sum\limits_{i\in S_0}\binom{3q-3}{i}a^{3q-3-i} \\
      &=a^{\frac{17q-16}{6}}+a^{\frac{5q}{2}-3}+a^{\frac{q-32}{6}}\\
      &=a^{\frac{q-32}{6}}(a^{\frac{8q+8}{3}}+a^{\frac{7q+7}{3}}+1)=0.
\end{aligned}
\end{eqnarray}

Therefore $a^{\frac{8q+8}{3}}+a^{\frac{7q+7}{3}}+1=0$. Let $y=a^{-\frac{q+1}{3}}$. 
Then $y\in \mu_{3(q-1)}$ and 
\begin{eqnarray}\label{32}
    y^8+y+1=0. 
\end{eqnarray}

It is easy to check that $\omega,\omega^2$ are solutions of $y^8+y+1=0$, where $\omega$ is a primitive $3$-th root of unity. In fact, there are no other solutions in $\mu_{3(q-1)}$. Because $y^8+y+1=0$ is an affine equation, all the solutions are of the form $\omega+b$, where $b\in \mathbb{F}_{2^3}$. If $3\nmid m$, then $\omega$, $\omega+1=\omega^2$ are the only two solutions in $\mathbb{F}_{2^{2m}}$. If $3\mid m$ and $\omega+b$ is a solution in $\mu_{3(q-1)}$, then we have $(\omega+b)^{3q}=(\omega+b)^3$. Obviously, 

\begin{equation}\label{33}
        (\omega+b)^3=1+\omega^2b+b^2\omega+b^3,
\end{equation}

\begin{eqnarray}\label{34}
    \begin{aligned}
        (\omega+b)^{3q}=(\omega^q+b)^3
                       =(\omega^2+b)^3
                       =1+b\omega+b^2\omega^2+b^3.
    \end{aligned}
\end{eqnarray}

From Equations (\ref{33}) and (\ref{34}), we have
\begin{eqnarray}\label{35}
    b^2\omega+b\omega^2=b\omega+b^2\omega^2.
\end{eqnarray}

From Equation (\ref{35}), we have $(b^2+b)\omega=(b^2+b)\omega^2$.  Hence $b^2+b=0$ and thus $b\in \{0,1\}$.  Therefore there are no other solutions in $\mu_{3(q-1)}$.  Hence $y = a^{-\frac{q+1}{3}} = \omega$ or $\omega^2$ and thus $a^{q+1} = 1$ but  $a^{\frac{q+1}{3}} \neq 1$. 

Conversely, by the AGW criterion, we have the following diagram:

\begin{center}
    \begin{tikzcd}
\mathbb{F}_{q^2}^{*} \arrow[d, "x^{q-1}"] \arrow[rr, "x(x^{6q-6}+a)"] &  & \mathbb{F}_{q^2}^{*} \arrow[d, "x^{q-1}"] \\
\mu_{q+1} \arrow[rr, "x(x^6+a)^{q-1}"]                                      &  & \mu_{q+1}                                    
\end{tikzcd}
\end{center}

From this diagram, we know $x(x^{6q-6}+a)$ is a permutation over $\mathbb{F}_{q^2}$ if and only if $x(x^6+a)^{q-1}$ permutes $\mu_{q+1}$. If $a\in \mu_{q+1}$ and $a\not \in \mu_{\frac{q+1}{3}}$, then $x^6+a$ has no root in $\mu_{q+1}$. In this case,  we have

\begin{eqnarray}\label{36}
    \begin{aligned}
        x(x^6+a)^{q-1}=\frac{x(x^{-6}+a^{-1})}{x^6+a}=a^{-1}x^{-5}.
    \end{aligned}
\end{eqnarray}

Because $n$ is odd, $\text{gcd}(5,2^n+1)=1$. Therefore $x(x^6+a)^{q-1}$ permutes $\mu_{q+1}$ and thus $x(x^{6q-6}+a)$ is a permutation binomial of $\mathbb{F}_{q^2}$.
\end{proof}

\subsection{Permutation binomials over $\mathbb{F}_{q^3}$}

Next we characterize another two new classes of permutation binomials over $\mathbb{F}_{q^3}$. These classes contain permutation binomials $x^{10}+ax$ and $x^{19}+ax$ over $\mathbb{F}_{2^6}$. 

\begin{thm}\label{T1}
	Let $q=2^n$, $n$ even, and $a\in \mathbb{F}_{q^3} ^{*}$. Then the polynomial 
	\[
	f_2(x)=x^{\frac{q^2+q}{2}}+ax
	\]
	is a permutation polynomial in $\mathbb{F}_{q^3}$ if and only if $a\in \mu_{q^2+q+1}$ and $a\not\in \mu_{\frac{q^2+q+1}{3}}$.
\end{thm}
\begin{proof}

  If $f_2(x)$ is a permutation polynomial, then by Hermite's criterion, the coefficient of $x^{q^3-1}$ in the reduction of $f_2(x)^t \mod (x^{q^3} -x)$ vanishes for any $1\leq t\leq q^3-2$.
Let $t=2q^2-q-1$. We have 

\begin{eqnarray}\label{1}
    \begin{aligned}
f_2(x)^t&=(x^{\frac{q^2+q}{2}}+ax)^{2q^2-q-1}\\
      &=x^{2q^2-q-1}\sum\limits_{i=0}^{2q^2-q-1} \binom{2q^2-q-1}{i}x^{\frac{(q^2+q-2)i}{2}}a^{2q^2-q-1-i}\\
      &=\sum\limits_{i=0}^{2q^2-q-1}\binom{2q^2-q-1}{i}a^{2q^2-q-1-i}x^{2q^2-q-1+\frac{(q^2+q-2)i}{2}}.
\end{aligned}
\end{eqnarray}

The coefficient of $x^{q^3-1}$ in Equation~(\ref{1}) is $\sum\limits_{i=0}^{2q^2-q-1}\binom{2q^2-q-1}{i}a^{2q^2-q-1-i}$ such that $(2q+1)+\frac{q+2}{2}i\equiv 0\hspace{0.5em}(\text{mod} \hspace{0.5em}(q^2+q+1))$, which is equivalent to $4q+2+(q+2)i\equiv 0\hspace{0.5em}(\text{mod} \hspace{0.5em}2(q^2+q+1))$. 


Let 
\begin{eqnarray}\label{2}
    (q+2)i+4q+2=2k(q^2+q+1),
\end{eqnarray}
where $k\in \mathbb{Z}$. Then


\begin{eqnarray}\label{3}
    i=k(2q-2)-4+\frac{6k+6}{q+2}.
\end{eqnarray}

Because $i$ is an integer, we obtain

\begin{eqnarray}\label{4}
    6k+6\equiv 0\hspace{0.5em}(\text{mod} \hspace{0.5em}q+2).
\end{eqnarray}

Because $6\mid (q+2)$, Equation (\ref{4}) is equivalent to

\begin{eqnarray}\label{5}
    k+1\equiv 0\hspace{0.5em}(\text{mod} \hspace{0.5em}\frac{q+2}{6}).
\end{eqnarray}

Because $i\leq 2q^2-q-1$, we obtain $k\leq q$. Together with Equation (\ref{5}), we know that $k$ can take up to five values: $\frac{q+2}{6}-1$, $\frac{q+2}{3}-1$, $\frac{q+2}{2}-1$, $\frac{2(q+2)}{3}-1$, $\frac{5(q+2)}{6}-1$. Plugging these values into Equation (\ref{3}), we obtain  $i \in S=\{\frac{q^2-5q-5}{3},\frac{2q^2-4q-4}{3}, q^2-q-1, \frac{4q^2-2q-2}{3}, \frac{5q^2-q-1}{3}\}$.

Hence the coefficient of $x^{q^3-1}$ is

\begin{eqnarray}\label{6}
   \sum\limits_{i\in S}\binom{2q^2-q-1}{i}a^{2q^2-q-1-i} =\sum\limits_{j\in S'}\binom{2q^2-q-1}{j}a^{j}, 
\end{eqnarray}
where $S'=\{q^2+\frac{2}{3}(q^2+q+1), q^2+\frac{1}{3}(q^2+q+1), q^2, q^2-\frac{1}{3}(q^2+q+1), q^2-\frac{2}{3}(q^2+q+1)\}$.

We note that $2q^2-q-1=q^2+\frac{q^2}{2}+\cdots+2q+\frac{q}{2}+\cdots+1$. From Lucas' lemma, $\binom{2q^2-q-1}{j}\neq 0$ if and only if the $(n+1)$ th bit of the binary expansion of $j$ is $0$. We denote by $S_0 \subseteq S'$ such that the $(n+1)$ th bit of the binary expansion of $j \in S'$ is $0$. Clearly $j=q^2 \in S_0$. 

When $q>4$, we first look at the $2$-adic expansion of $\frac{q^2+q+1}{3}=\frac{1}{3}(q^2+q-2)+1=\frac{1}{3}(q-1)(q+2)+1$.  Because $q=2^n$ and $n$ is even, we can easily determine $\frac{q-1}{3}$'s $2$-adic expansion, that is $\frac{q-1}{3}=\frac{q}{4}+\frac{q}{16}+\frac{q}{64}+\cdots+1$. Using $\frac{q^2+q+1}{3}=\frac{1}{3}(q-1)(q+2)+1$, we obtain $\frac{1}{3}(q^2+q+1)= \frac{q^2}{4}+\frac{q^2}{16}+\frac{q^2}{64}+\cdots+q+\frac{q}{2}+\frac{q}{8}+\cdots+2+1$.

 Therefore the $(n+1)$ th bit of $q^2+\frac{1}{3}(q^2+q+1)$ and $q^2+\frac{2}{3}(q^2+q+1)$ is not $0$ respectively. Hence both $q^2+\frac{2}{3}(q^2+q+1)$ and $q^2+\frac{1}{3}(q^2+q+1)$  do not belong to  $S_0$. 

Similarly, we obtain $q^2-\frac{1}{3}(q^2+q+1)=\frac{q^2}{2}+\frac{q^2}{8}+\frac{q^2}{32}+\cdots+2q+\frac{q}{4}+\frac{q}{16}+\cdots+1$ and $q^2-\frac{2}{3}(q^2+q+1)=\frac{q^2}{4}+\frac{q^2}{16}+\frac{q^2}{64}+\cdots+\frac{q}{2}+\frac{q}{8}+\cdots+2$. Hence $j=q^2-\frac{1}{3}(q^2+q+1), q^2-\frac{2}{3}(q^2+q+1) \in S_0$. It is easy to verify when $q=4$, $q^2-\frac{1}{3}(q^2+q+1)$, $q^2-\frac{2}{3}(q^2+q+1)\in S_0$. Hence we know $S_0=\{q^2, q^2-\frac{1}{3}(q^2+q+1), q^2-\frac{2}{3}(q^2+q+1)\}$.

Therefore the coefficient of $x^{q^3-1}$ is

\begin{eqnarray}\label{8}
\begin{aligned}
   \sum\limits_{j\in S'}\binom{2q^2-q-1}{j}a^{j} &=a^{q^2}+a^{q^2-\frac{1}{3}(q^2+q+1)}+a^{q^2-\frac{2}{3}(q^2+q+1)}\\
    &=a^{q^2-\frac{2}{3}(q^2+q+1)}(a^{\frac{2}{3}(q^2+q+1)}+a^{\frac{1}{3}(q^2+q+1)}+1).\\
\end{aligned}
\end{eqnarray}

If $a\in \mu_{\frac{q^2+q+1}{3}}$, then the coefficient of $x^{q^3-1}$ is 
$a^{q^2-\frac{2}{3}(q^2+q+1)} \neq 0$, contradicts the fact that $f_2(x)$ is a permutation binomial. 
Hence $a\not\in \mu_{\frac{q^2+q+1}{3}}$.  In this case, we must have 
$a^{q^2+q+1}=1$.

Conversely, it suffices to prove $f_2(x)=b$ has at most one solution for all $b\in \mathbb{F}_{q^3}$ when $a\in \mu_{q^2+q+1}$ and $a\not\in \mu_{\frac{q^2+q+1}{3}}$. Let $g_2(x)=f_2(x^2)$.  Obviously $f_2(x)$ is a permutation polynomial of $\mathbb{F}_{q^3}$ if and only if $g_2(x)$ is a permutation polynomial  of $\mathbb{F}_{q^3}$. Therefore we will show that $g_2(x) =b$ has at most one solution for any $b \in \mathbb{F}_{q^3}$.  If $b=0$, then $a\not \in \mu_{\frac{q^2+q+1}{3}}$ implies that $x^{q^2+q} + a x^2=0$ has only zero solution. 

From now on, we assume $b\neq 0$. Let $y=x^q$, $z=y^q$, $b_1=b$, $b_2=b^q$, $b_3=b_2^q$.  Raising $q$-th power to $g_2(x)=b$,  we obtain a system of equations

\begin{eqnarray}\label{11}
	\left\{\begin{array}{c}
		yz+ax^2=b_1,\\
        zx+a^qy^2=b_2,\\
        xy+a^{q^2}z^2=b_3.
	\end{array}
	\right.
\end{eqnarray}

Since $b_1=b\neq 0$, we must have $x\neq 0$. From the second equation of (\ref{11}), we write

\begin{eqnarray}\label{12}
    z=\frac{b_2+a^{q}y^2}{x}.
\end{eqnarray}

Plugging Equation (\ref{12}) into the first and third equation of (\ref{11}), we get

\begin{eqnarray}\label{13}
	\left\{\begin{array}{c}
       h_1(x,y)=ax^3+b_1x+a^qy^3+b_2y=0,\\
       h_2(x,y)=x^3y+b_3x^2+a^{q^2+2q}y^4+a^{q^2}b_2^2=0.
\end{array}
	\right.
\end{eqnarray}

The resultant of Equation (\ref{13}) with respect to the indeterminate $y$ is

\begin{eqnarray}\label{14}
    (a^{3q-2}b_1^3+a^{3q-1}b_1b_2b_3+a^{2q-1}b_2^3+a^{4q}b_3^3)x^6+(a^{3q-3}b_1^4+a^{3q-1}b_2^2b_3^2)x^4=0.
\end{eqnarray}

Let $A=a^{3q-2}b_1^3+a^{3q-1}b_1b_2b_3+a^{2q-1}b_2^3+a^{4q}b_3^3$ and $B=a^{3q-3}b_1^4+a^{3q-1}b_2^2b_3^2$. If exactly one of $A$ and $B$ is zero,  then $x$ is zero, a contradiction. 
Next we show both $A$ and $B$ are not zero. 
Assume, to the contrary, that $A=B=0$. We can get a system of equations

\begin{eqnarray}\label{15}
\left\{\begin{array}{c}
     a^{3q-2}b_1^3+a^{3q-1}b_1b_2b_3+a^{2q-1}b_2^3+a^{4q}b_3^3=0,\\
   a^{3q-3}b_1^4+a^{3q-1}b_2^2b_3^2=0.
\end{array}
	\right.
\end{eqnarray}

From the second equation of (\ref{15}), we get $b_1^2=ab_2b_3$, which is equivalent to saying $a^{-1}=b_1^{q^2+q-2}$. But $\text{gcd}(q^2+q-2,q^3-1)=3(q-1)$ implies $a\in \mu_{\frac{q^2+q+1}{3}}$, a contradiction. Hence $AB\neq 0$ and thus Equation (\ref{14}) has at most one non zero solution for each $b_1$. Hence $g_2(x) =b $ has at most one
solution for any $b\in \mathbb{F}_{q^3}$.  
\end{proof}

\begin{rmk}
    In \cite{SNROPPOFF2017}, the authors constructed a class of complete permutation monomials: $x^{2^{4k-1}+2^{2k-1}}$ over $\mathbb{F}_{2^{6k}}$, where $k$ is a positive integer with \text{gcd}$(k,3)=1$. Our theorem generalizes their result.
\end{rmk}

Because $(q^2+q-1)\frac{q^2+q}{2}\equiv 1\hspace{0.5em}(\text{mod}\hspace{0.5em}(q^3-1))$, we have the following  permutation polynomial  $h_2(x) = a^{-1} g_2(x^{q^2+q-1})$. 

\begin{cor}\label{cor1}
    Let $q=2^n$, $n$ even, and $a\in \mathbb{F}_{q^3} ^{*}$. Then the polynomial 
	\[
	h_2(x)=x^{q^2+q-1}+ax
	\]
	is a permutation polynomial in $\mathbb{F}_{q^3}$ if and only if $a\in \mu_{q^2+q+1}$ and $a\not\in \mu_{\frac{q^2+q+1}{3}}$.
\end{cor}

\begin{rmk}
   Theorem \ref{T1} and Corollary \ref{cor1} give us two classes of permutation binomials which have relatively large  index $\frac{q^2+q+1}{3}$.
\end{rmk}

\subsection{Permutation binomials over $\mathbb{F}_{q^4}$}

To explain  the permutation behavior  of $x^{154}+ax$ in $\mathbb{F}_{2^8}$, we generalize it into an infinite class of permutation binomials over  $\mathbb{F}_{q^4}$.

\begin{thm}\label{T2}
 Let $q=2^n$, $n>1$  and $a\in \mathbb{F}_{q^4}^{*}$. The polynomial 
 \[
f_3(x)=x^{\frac{q^3-q^2+q-1}{2}+1}+ax
 \]
 is a permutation polynomial in $\mathbb{F}_{q^4}$ if and only if $a\in \mu_{q^2-1}$ and $a\not\in \mu_{q+1}$.
\end{thm}

\begin{proof}
    Let $g_3(x)=f_3(x^2)$.  Obviously $f_3(x)$ is a permutation polynomial if and only if $g_3(x)$ is a permutation polynomial. Next we show $g_3(x)$ is a permutation polynomial if and only if $a\in \mu_{q^2-1}$ and $a\not\in \mu_{q+1}$.

    By the AGW criterion, we have the following diagram:
 \begin{center}
     \begin{tikzcd}
\mathbb{F}_{q^4}^{*} \arrow[rrrr, "g_3(x)= x^2(x^{(q-1)(q^2+1)}+a)"] \arrow[dd, "x^{(q-1)(q^2+1)}"] &  &  &  & \mathbb{F}_{q^4}^{*} \arrow[dd, "x^{(q-1)(q^2+1)}"] \\
                                                                                        &  &  &  &                                                 \\
\mu_{q+1} \arrow[rrrr, "h_3(x) = x^2(x+a)^{(q-1)(q^2+1)}"]                                       &  &  &  & \mu_{q+1}                                      
\end{tikzcd}
 \end{center}

Because $\text{gcd}(2,(q-1)(q^2+1))=1$, $g_3(x)$ is a permutation polynomial if and only if $h_3(x)=x^2(x+a)^{(q-1)(q^2+1)}$ permutes $\mu_{q+1}$.  Note $\text{gcd}((q-1)(q^2+1),q^4-1)=(q-1)(q^2+1)$. If  $a\in \mu_{q+1}$, then $x^{(q-1)(q^2+1)}+a$ has nonzero roots and thus $g_3(x)$ can not be a permutation polynomial.

Next we study when $h_3(x)=x^2(x+a)^{(q-1)(q^2+1)}$ permutes $\mu_{q+1}$.   
Consider: $\phi(x): \mathbb{F}_{q}\cup \{\infty\}\rightarrow\mu_{q+1}$ such that $\phi(x)=\frac{x+z}{x+z^q}$, where $z\in \mathbb{F}_{q^2}\setminus \mathbb{F}_{q}$, and $\phi^{-1}(x): \mu_{q+1}\rightarrow\mathbb{F}_q\cup \{\infty\}$ such that $\phi^{-1}(x)=\frac{z^qx+z}{x+1}$. It is easy to check $\phi^{-1}(\phi(x))=x$. Then $h_3(x)$ permutes $\mu_{q+1}$ if and only if $l_3(x)=\phi^{-1}(h_3(\phi(x))$ permutes $\mathbb{F}_q\cup \{\infty\}$. Here is the diagram illustrating this relation.
\begin{center}
    \begin{tikzcd}
\mu_{q+1} \arrow[rr, "h_3(x)"]                                         &  & \mu_{q+1} \arrow[d, "\phi^{-1}(x)"] \\
\mathbb{F}_q\cup \{\infty\} \arrow[u, "\phi(x)"'] \arrow[rr, "l_3(x)"] &  & \mathbb{F}_q\cup \{\infty\}        
\end{tikzcd}
\end{center}

Now let us compute $l_3(x)$. First, we compute $h_3(\phi(x))$, that is

\begin{eqnarray}\label{16}
\begin{aligned}
      h_3(\phi(x)) &= \bigg(\frac{x+z}{x+z^q}\bigg)^2 \frac{\left(\bigg(\frac{x+z}{x+z^q}\bigg)^2 +a\right)^{q^3+q}}{\left(\bigg(\frac{x+z}{x+z^q}\bigg)^2 +a\right)^{q^2+1}} \\
      &=\frac{a^{q^3+q}\bigg(\frac{x+z}{x+z^q}\bigg)^2+(a^q+a^{q^3})\bigg(\frac{x+z}{x+z^q}\bigg)+1}{\bigg(\frac{x+z}{x+z^q}\bigg)^2+(a+a^{q^2})\bigg(\frac{x+z}{x+z^q}\bigg)+a^{q^2+1}}\\
      &=\frac{a^{q^3+q}(x+z)^2+(a^{q}+a^{q^3})(x+z)(x+z^q)+(x+z^q)^2}{(x+z)^2+(a+a^{q^2})(x+z)(x+z^q)+a^{q^2+1}(x+z^q)^2}\\
      &=\frac{(a^{q}+1)(a^{q^3}+1)x^2+(a^q+a^{q^3})(z+z^q)x+a^{q^3+q}z^2+(a^{q}+a^{q^3})z^{q+1}+z^{2q}}{(a^{q^2}+1)(a+1)x^2+(a+a^{q^2})(z+z^q)x+z^2+(a+a^{q^2})z^{q+1}+a^{q^2+1}z^{2q}}.
\end{aligned}
\end{eqnarray}

This gives 

\begin{eqnarray}\label{17}
\begin{aligned}
     l_3(x)&=\phi^{-1}(h_3(\phi(x)))\\
           &=\frac{z^q\frac{(a^{q}+1)(a^{q^3}+1)x^2+(a^q+a^{q^3})(z+z^q)x+a^{q^3+q}z^2+(a^{q}+a^{q^3})z^{q+1}+z^{2q}}{(a^{q^2}+1)(a+1)x^2+(a+a^{q^2})(z+z^q)x+z^2+(a+a^{q^2})z^{q+1}+a^{q^2+1}z^{2q}}+z}{\frac{(a^{q}+1)(a^{q^3}+1)x^2+(a^q+a^{q^3})(z+z^q)x+a^{q^3+q}z^2+(a^{q}+a^{q^3})z^{q+1}+z^{2q}}{(a^{q^2}+1)(a+1)x^2+(a+a^{q^2})(z+z^q)x+z^2+(a+a^{q^2})z^{q+1}+a^{q^2+1}z^{2q}}+1}\\
           &=\frac{Ax^2+Bx+C}{Dx^2+Ex+F}\in \mathbb{F}_q(x),
\end{aligned}
\end{eqnarray}
where 
\begin{eqnarray}\label{18}
    A=z^q(a^q+1)(a^{q^3}+1)+z(a^{q^2}+1)(a+1),
\end{eqnarray}
\begin{eqnarray}\label{19}
    B=z^q(a^q+a^{q^3})(z+z^q)+z(a+a^{q^2})(z+z^q),
\end{eqnarray}
\begin{eqnarray}\label{20}
    C=z^q(a^{q^3+q}z^2+(a^q+a^{q^3})z^{q+1}+z^{2q})+z(a^{q^2+1}z^{2q}+(a+a^{q^2})z^{q+1}+z^2),
\end{eqnarray}
\begin{eqnarray}\label{21}
    D=(a^q+1)(a^{q^3}+1)+(a^{q^2}+1)(a+1),
\end{eqnarray}
\begin{eqnarray}\label{22}
    E=(a^q+a^{q^3})(z+z^q)+(a+a^{q^2})(z+z^q),
\end{eqnarray}
\begin{eqnarray}\label{23}
    F=a^{q^3+q}z^2+(a^q+a^{q^3})z^{q+1}+z^{2q}+z^2+(a+a^{q^2})z^{q+1}+a^{q^2+1}z^{2q}.
\end{eqnarray}

If $a\in \mu_{q^2-1}$, then $l_3 (x) = \frac{ ((a^q+1)^2z^q+(a+1)^2z)  x^2 + z^q(a^{2q}z^2 + z^{2q}) + z(a^2z^{2q} + z^2)}{((a^q+1)^2 +(a+1)^2 )x^2 + (a^{2q}z^2 + z^{2q} + z^2 + a^2z^{2q})}$ is a composition of degree-one rational function $\frac{Gx+H}{Mx+N}$ and $x^2$, 
where 
\begin{eqnarray}\label{37}
    G=(a^q+1)^2z^{q}+(a+1)^2z,
\end{eqnarray}
\begin{eqnarray}\label{38}
    H=z^q(a^{2q}z^2+z^{2q})+z(a^2z^{2q}+z^2),
\end{eqnarray}
\begin{eqnarray}\label{39}
    M=(a^q+1)^2+(a+1)^2,
\end{eqnarray}
\begin{eqnarray}\label{40}
    N=a^{2q}z^2+z^{2q}+z^2+a^2z^{2q}.
\end{eqnarray}

Next we show $\frac{Gx+H}{Mx+N}$ is invertible over $\mathbb{F}_q(x)$. It suffices to check $GN-HM\neq 0$. Indeed,
\begin{eqnarray}\label{41}
    \begin{aligned}
        GN-HM&=(L^{q}z^{q}+Lz)(L^{q}z^2+Lz^{2q})+(L+L^{q})(a^2z^{2q+1}+z^3+z^{3q}+a^{2q}z^{q+2})\\
             &=(L^{q+1}+L^{q}+L)z^{3q}+(L^2+a^2(L+L^{q}))z^{2q+1}\\
             &\hspace{1em}+(L^{2q}+a^{2q}(L+L^{q}))z^{q+2}+(L^{q+1}+L^{q}+L)z^3,\\
    \end{aligned}
\end{eqnarray}
where $L=(a+1)^2$.


From Equation (\ref{41}), we know $(GN-HM)(z)$ is a polynomial, whose degree is $3q$.  We note that we can choose any $z\in \mathbb{F}_{q^2}\setminus\mathbb{F}_{q}$, whose cardinality is $q^2-q$. When $q^2-q>3q$, i.e. $q>4$, if all such $z$'s are zeroes of $GN-HM$, then $(GN-HM)(z)$ is a zero polynomial by the fundamental theorem of algebra. From the fact that $(GN-HM)(z)$ is a zero polynomial, we have
\begin{eqnarray}\label{43}
	\left\{\begin{array}{c}
      a^2(L+L^{q})=L^2,\\
      L+L^{q}=L^{q+1}.
\end{array}
	\right.
\end{eqnarray}

Plugging $L=(a+1)^2$ into the second equation of (\ref{43}), we get $a^{2(q+2)}=1$, which is a contradiction because $a\not \in \mu_{q+1}$. Thererfore $GN-HM\neq 0$ and thus $l_3(x)$ permutes $\fq\cup\{\infty\}$ for $q>4$. Hence $g_3(x)$ is a permutation polynomial when $q>4$.  The cases for small $q$'s can be verified directly or by a computer search.

Conversely, we first prove at least one of  $A$ and $D$ in Equation~(\ref{17}) is nonzero . Assume both $A$ and $D$ are $0$, we have
\begin{eqnarray}\label{44}
	\left\{\begin{array}{c}
     z^{q}(a^{q}+1)(a^{q^3}+1)+z(a^{q^2}+1)(a+1)=0,\\
     (a^{q}+1)(a^{q^3}+1)+(a^{q^2}+1)(a+1)=0.\\
\end{array}
	\right.
\end{eqnarray}

From Equation (\ref{44}) and $z\not \in \mathbb{F}_{q}$, we have $a=1$, which is a contradiction. It is well known that  a degree-two rational function $f(x)\in \mathbb{F}_{q}(x)$ permutes $\mathbb{P}^{1}(\mathbb{F}_q)$ if and only if $q$ is even and $f(x)$ is equivalent to $x^2$. Specifically, there exist some degree one $\mu,\nu\in \mathbb{F}_{q}(x)$ such that $f=\mu\circ x^2\circ \nu$. 
Applying this result to the permutation rational function $l_3(x)$, we must have $B=E=0$. From $E=0$, we get
\begin{eqnarray}\label{24}
    a+a^{q^2}=a^{q}+a^{q^3}.
\end{eqnarray}

Plugging it into the equation $B=0$, we get
\begin{eqnarray}\label{25}
    (a+a^{q^2})(z+z^q)^2=0.
\end{eqnarray}

From Equation (\ref{25}), we get $a=a^{q^2}$. In other words, $a\in \mu_{q^2-1}$, because $a$ is nonzero.  
\end{proof}


Finally we prove   a nonexistence result over $\mathbb{F}_{q^4}$. 

\begin{thm}\label{T4}
    Let $q=2^n$, $n$ even, and $n\geq 4$. Then the polynomial

    \[
    f_4(x)=x^{2q^3+2q^2+2q+3}+ax
    \] 
is not a permutation polynomial over $\mathbb{F}_{q^4}$ for any $a\in \mathbb{F}_{q^4}^{*}$.
\end{thm}

\begin{proof}
    By the AGW criterion, we have the following diagram:

    \begin{center}
\begin{tikzcd}
\mathbb{F}_{q^4}^{*} \arrow[rrrr, "f_4(x)=x^{2q^3+2q^2+2q+3}+ax"] \arrow[dd, "x^{q^3+q^2+q+1}"] &  &  &  & \mathbb{F}_{q^4}^{*} \arrow[dd, "x^{q^3+q^2+q+1}"] \\
                                                                                        &  &  &  &                                                 \\
\mathbb{F}_q^* \arrow[rrrr, "g_4(x)=x(x^2+a)^{q^3+q^2+q+1}"]                                       &  &  &  & \mathbb{F}_q^*                                    
\end{tikzcd}
    \end{center}

Therefore $f_4(x)=x^{2q^3+2q^2+2q+3}+ax$ is a permutation polynomial if and only if $g_4(x)=x(x^2+a)^{q^3+q^2+q+1}$ permutes $\mathbb{F}_{q}^{*}$.

 Let us consider $G(x)=x^9+Ax^7+Bx^5+Cx^3+Dx$, where $A, B,C,D\in \mathbb{F}_q$. Assume $G(x)$ is a permutation polynomial over $\mathbb{F}_{q}$.  Since $\text{gcd}(9,q-1)>1$, by the Carlitz-Wan conjecture, $G(x)$ is not exceptional. By Lemma~\ref{PPEP},  if  $q^{\frac{1}{4}}\geq 9$, i.e. $q\geq 6561$, then $G(x)$ is exceptional, which is impossible. Hence $q< 6561$.  By Hermite's criterion, $G(x)$ is not a permutation polynomial over $\mathbb{F}_{4096}$ since $9\mid 4095$. As for $4\leq m<12$, by a computer search, no permutation polynomials of the form $x^{2q^3+2q^2+2q+3}+ax$ exist. Hence we complete the proof. 
\end{proof}

By a computer search, for some $a\in \mathbb{F}_{2^8}$, $x^{171}+ax$ can be a permutation binomial. Hence we have the following corollary:

\begin{cor}\label{cor2}
    Let $q=2^n$, $n$ even. Then the polynomial $x^{2q^3+2q^2+2q+3}+ax$ can be a permutation binomial over $\mathbb{F}_{q^4}$ only when $q=4$. 
\end{cor}
\section{Discussion}

With these newly constructed permutation binomials and previous results in the literature, we can classify all permutation binomials of the form  $x^i+ax$ over $\mathbb{F}_{2^n}$, where $n\leq 8$. 

For any linearized binomial $x^{2^j} + ax$, it is a PP over $\mathbb{F}_{2^n}$  if and only if it has only trivial root. Hence $x^{2^j-1} + a \neq 0$.  This is equivalent to say that  $a$ is not a ($2^{\gcd(j, n)}-1$)-th power in $\mathbb{F}_{2^n}$. Hence we only need to focus on the non-linearizd PPs over $\mathbb{F}_{2^n}$. By a computer search, when $n=4$, only linearized permutation polynomials exist. For this reason, the smallest case that allows the non-linearized PPs of this form is $n=6$. 

We note that  there are no permutation binomials of the form $x^i+ax$ in $\mathbb{F}_{2^5}$, $\mathbb{F}_{2^7}$ respectively. Indeed, using Corollary~3 in \cite{TNOFPB2006},  we can explain that there  is no permutation binomial  over $\mathbb{F}_{2^5}$ or $\mathbb{F}_{2^7}$ because $31,127$ are Mersenne primes. 

In $\mathbb{F}_{2^6}$, we can use Theorem \ref{T1} and Corollary \ref{cor1} or \cite{PPIFPOSASCSOMOLS2013} to explain when $x^{10}+ax$ and $x^{19}+ax$ can be permutation binomials. As for $x^{22}+ax$, we use the result of \cite{DOATOPBOFF2015} to explain. As for the last polynomial $x^{43}+ax$ in Table \ref{Table1}, we can use Theorem \ref{T3} to explain its permutation behavior.

In $\mathbb{F}_{2^8}$, we can use the result in  \cite{OOCOPPOFFCT2015} to explain when  $x^{86}+ax$ is a permutation binomial. As for $x^{154}+ax$ and $x^{171}+ax$, we can use Theorem \ref{T2} and Corollary \ref{cor2} to demonstrate that they are indeed permutation binomials. Thus we have provided a full explanation  of permutation binomials of the form $x^i+ax$ over $\mathbb{F}_{2^n}$, where $n\leq 8$.

\begin{rmk}
    As for $x^{366}+ax$ over $\mathbb{F}_{2^9}$, from the AGW criterion, we know $x^{366}+ax$ is a permutation binomial over $\mathbb{F}_{2^9}$ if and only if $x(x^5+a^{64})(x^5+a^8)(x^5+a)$ permutes $\mathbb{F}_{2^3}$. Using $5\times 3\equiv 1\pmod 7$, the latter is equivalent to $x^3(x+a^{64})(x+a^8)(x+a)$ permutes $\mathbb{F}_{2^3}$. Hence, we can use the classification of permutation polynomial of degree $6$ \cite{PPODOFFOC2010,PPAOPODS2013} or Hermite's criterion to obtain conditions on the coefficient $a$ so that $x^{366}+ax$ permutes $\mathbb{F}_{2^9}$.  Indeed, $x^{366}+ax$ permutes $\mathbb{F}_{2^9}$ if and only if $a\not\in \mathbb{F}_{2^3}$, $(a^2+a^{16}+a^{128})(a^{65}+a^{9}+a^{72})+(a^{130}+a^{18}+a^{144})=0$ and $(a^{65}+a^{9}+a^{72})+(a^{260}+a^{36}+a^{288})(a+a^{8}+a^{64})=0$, or $a^{65}+a^9+a^{72}=0$ and $\frac{a^{73}}{(a+a^{8}+a^{64})^3}$ is a root of $x^3+x+1$ over $\mathbb{F}_{2^9}$. Similarly, when $3e\equiv 1\pmod{(q-1)}$, using the classification of permutation polynomial of degree $6$, we can get similar conditions on $a$ so that $x(x^{e(q^2+q+1)}+a)$ permutes $\mathbb{F}_{q^3}$. 
\end{rmk}

\section{Conclusion}

In this paper, we gave a full explanation of permutation binomials of the form $x^i+ax$ over $\mathbb{F}_{2^n}$ where $n\leq 8$, using new and known classes of PPs. In particular, we gave three classes of permutation binomials over $\mathbb{F}_{2^n}$ and a nonexistence result. It would be interesting to follow this path to complete the classification for arbitrary $n$, as well as give a classification of the form $x^i+ax$ over $\mathbb{F}_{p^n}$, where $p$ is an odd prime number.

\bibliographystyle{elsarticle-num}\biboptions{sort&compress,longnamesfirst}
\bibliography{ffa-refs}
\end{document}